\documentclass{amsart}
\usepackage[hmarginratio=1:1]{geometry}

\usepackage{amsmath,amssymb,mathrsfs}

\usepackage{amsthm}

\usepackage{xcolor}
\definecolor{darkred}{rgb}{0.75,0,0}
\def\customcitecolor{darkred}
\def\customlinkcolor{darkred}
\usepackage[%
    colorlinks,
    citecolor=\customcitecolor,%
    linkcolor=\customlinkcolor,%
    urlcolor=\customlinkcolor
]{hyperref}

\usepackage{enumitem}

\usepackage{geometry}
\usepackage{microtype}

\usepackage{etoc}
\newcommand{\secttoc}{%
    \etocsettocstyle{%
        \par\medskip\noindent\textbf{\contentsname.}\par\smallskip
    }{\medskip}%
    \etocsetnexttocdepth{subsection}%
    \localtableofcontents
}

\usepackage{quiver}

\usepackage{color}

\ifdefined\skipbibliography
\else
    \usepackage[style=alphabetic]{biblatex}
    \DeclareFieldFormat{postnote}{\mknormrange{#1}}
    \DeclareFieldFormat{volcitepages}{\mknormrange{#1}}
    \DeclareFieldFormat{multipostnote}{\mknormrange{#1}}
\fi

\usepackage{mathtools}

\usepackage{cleveref}

\providecommand{\xto}[1]{\xrightarrow{#1}}

\theoremstyle{definition}
\newtheorem{theorem}{Theorem}[section]

\AddToHook{env/conjecture/begin}{\crefalias{theorem}{conjecture}}

\newtheorem{corollary}[theorem]{Corollary}
\AddToHook{env/corollary/begin}{\crefalias{theorem}{corollary}}

\newtheorem{definition}[theorem]{Definition}
\AddToHook{env/definition/begin}{\crefalias{theorem}{definition}}

\AddToHook{env/terminology/begin}{\crefalias{theorem}{terminology}}

\newtheorem{example}[theorem]{Example}
\AddToHook{env/example/begin}{\crefalias{theorem}{example}}

\AddToHook{env/note/begin}{\crefalias{theorem}{note}}

\newtheorem{lemma}[theorem]{Lemma}
\AddToHook{env/lemma/begin}{\crefalias{theorem}{lemma}}

\newtheorem{notation}[theorem]{Notation}
\AddToHook{env/notation/begin}{\crefalias{theorem}{notation}}

\newtheorem{proposition}[theorem]{Proposition}
\AddToHook{env/proposition/begin}{\crefalias{theorem}{proposition}}

\AddToHook{env/question/begin}{\crefalias{theorem}{question}}

\newtheorem{remark}[theorem]{Remark}
\AddToHook{env/remark/begin}{\crefalias{theorem}{remark}}

\AddToHook{env/slogan/begin}{\crefalias{theorem}{slogan}}

\AddToHook{env/assumption/begin}{\crefalias{theorem}{assumption}}

\newcommand\categoryfont\mathrm

\newcommand{\Et}{\categoryfont{Et}}
\newcommand{\FEt}{\categoryfont{F\acute{E}t}}
\newcommand{\Field}{\categoryfont{Field}}
\newcommand{\Fin}{\categoryfont{Fin}}

\newcommand{\Set}{\categoryfont{Set}}

\newcommand{\prelring}{\categoryfont{Pre{\lambda}Ring}}

\newcommand{\Unimod}{\categoryfont{Unimod}}

\newcommand{\op}{\mathrm{op}}

\newcommand{\fps}[1]{[\![#1]\!]}
\newcommand{\Z}{\mathbb{Z}}
\newcommand{\C}{\mathbb{C}}

\DeclareMathOperator{\GW}{GW}
\DeclareMathOperator{\Gal}{Gal}

\DeclareMathOperator{\Inv}{Inv}
\DeclareMathOperator{\KO}{KO}
\DeclareMathOperator{\Map}{Map}

\DeclareMathOperator{\Nat}{Nat}

\DeclareMathOperator{\Sym}{Sym}

\DeclareMathOperator{\tr}{tr}
\DeclareMathOperator{\Tr}{Tr}

\renewcommand\emptyset\varnothing

\NewDocumentCommand{\VW}{}{V\!\!W}

\let\epsilon\varepsilon

\providecommand{\Inv}{\mathrm{Inv}}

\title{Operations on $G$-sets and the Dress Map}
\author{Thomas Brazelton}
\author{Ben Spitz}

\begin{document}
\begin{abstract}
    We study versions of binomial and multichoose coefficients for finite $G$-sets, extending a result of Siebeneicher from the 1970s concerning the duality between these operations. If the group $G$ comes from a finite Galois extension of fields, the Dress map provides a bridge from $G$-sets to quadratic forms over the base field. We show that this Dress map is compatible with binomial coefficients of finite $G$-sets and a particular operation on quadratic forms we define. This allows us to extend the definition of Brugall\'e and Wickelgren's quadratically enriched binomial coefficients to arbitrary Grothendieck--Witt classes. This allows for a new proof that quadratically enriched binomial coefficients depend only on the trace form of the inputted \'etale algebra, as well as an elegant reproof of a result of Chen and Wickelgren about binomial coefficients over finite fields.
\end{abstract}

\maketitle

\noindent{\;\;\;\;\small\textbf{Contents}\;\hrulefill\;\textbf{Page}}
\contentsline {section}{\numberline {1}\;\;\;\;Operations on $G$-sets}{\pageref{sec:operations}}{section.1}%
\contentsline {section}{\numberline {2}\;\;\;\;The Dress Map}{\pageref{sec:dress}}{section.2}%
\hrulefill

\section*{Introduction}

Central to the combinatorial theory of finite sets are the operations of ``choosing without replacement'' and ``choosing with replacement''. Formulaically, we know these operations by the name of binomial coefficients and multiset coefficients, respectively. These operations are dual in a formal sense, and underlie a number of other dualities in mathematics, including the duality between symmetric and exterior powers (of vector spaces, representations, etc.). These operations can also be explored in the theory of \emph{equivariant combinatorics}, i.e. the combinatorial theory of finite $G$-sets (or more generally, combinatorial structures with $G$-action) where $G$ is a (finite) group. In the equivariant setting, the aforementioned duality is more subtle, and was well-studied in the 20th century \cite{Knutson,Siebeneicher,Laitinen}. In Section 1, we introduce this theory and its connections with contemporary ideas in equivariant algebra.

In Section 2, we turn our attention to the \emph{Dress map}, which, for a finite $G$-Galois extension of fields $L/k$, is a morphism of $E_\infty$-rings from the $G$-fixed points of the $G$-equivariant sphere spectrum to the Hermitian $K$-theory spectrum of $k$:
\[D_{L/k} : \mathbb{S}_G^G \to \KO(k).\]
We construct this map at the categorical level, and obtain $D_{L/k}$ by applying $K$-theory. Dress first developed and explored the properties of this map on $\pi_0$, where its domain is the \emph{Burnside ring} $A(G) = \pi_0(\mathbb{S}_G^G)$ and its codomain is the \emph{Grothendieck--Witt ring} $\GW(k) = \pi_0 \KO(k)$. The ring map $\pi_0(D_{L/k})\colon A(G) \to \GW(k)$ is an important fixture in the connection between equivariant and motivic homotopy theory \cite{hellerGaloisEquivarianceStable2015,hellerStableGaloisCorrespondence2018}.

Recent work in the theory of quadratically enriched Gromov--Witten theory has led to interest in structures on $\GW(k)$ akin to binomial coefficients \cite{ChenWickelgren,BrugalleWickelgren}. We exhibit $\pi_0 D_{L/k}$ as a pre-$\lambda$ ring homomorphism, between the ``binomial'' pre-$\lambda$ ring structure on the Burnside ring, and a new pre-$\lambda$ ring structure on the Grothendieck--Witt ring we define. As an immediate corollary, we provide a new proof that quadratically enriched binomial coefficients are invariants only of the trace form of an \'etale $k$-algebra, not of the \'etale $k$-algebra itself. This allows for a lightweight reproof of Theorem 1.1 in \cite{ChenWickelgren}, and a new basis for the ring of Witt invariants in which to express Welschinger--Witt invariants of the projective plane (\Cref{thm:Bjs-basis-of-inv-gw}).

\subsection*{Acknowledgements}

Thanks to Jesse Pajwani and Anubhav Nanavaty for helpful conversations regarding motivic Euler characteristics. Thanks to Kirsten Wickelgren for providing crucial references and feedback on a draft of this article. Thanks to the Reed College REU, Ang\'elica Osorno, and Kyle Ormsby for hosting the authors on a visit where some of this research was performed. The authors also thank Kyle Ormsby for suggesting \Cref{def:bt-power-structure} and Ang\'elica Osorno and Mike Mandell for helpful references on $K$-theory. The first named author was partially supported by NSF award
DMS-230324.

\section{Operations on \texorpdfstring{$G$}{G}-sets}\label{sec:operations}
\secttoc

Combinatorics arises from the attempt to quantify various operations that can be carried out on finite sets of objects: for instance choosing $k$ objects from a set of $n$ objects, allowing or disallowing repeats. There is a formal duality between these two procedures, encoded very elegantly in a duality between elementary symmetric and elementary homogeneous symmetric polynomials.

When the finite sets in question are equipped with additional symmetry coming from an action of a finite group, this duality is more mysterious and less understood. Just as finite sets create the integers by taking isomorphism classes and group completing, finite $G$-sets create a commutative ring called the \emph{Burnside ring}, denoted $A(G)$. Exterior and symmetric operations of finite $G$-sets are each naturally framed as an additional structure on the Burnside ring, known as a \emph{pre-$\lambda$ structure}. These more general operations originate in algebraic $K$-theory and representation theory, where they are intended to mimic operations like binomials or symmetric powers \cite{Grothendieck-Chern,SGA6,AtiyahTall}. We situate the exterior and symmetric duality in the general language of pre-$\lambda$ rings, focusing on the Burnside ring of a finite group. Our main result in this section (\Cref{thm:Sieb-extension}) is an extension of Siebeneicher's theorem (\Cref{thm:Siebeneicher}) which explores for which finite groups the exterior and symmetric pre-$\lambda$ structures on the Burnside ring $A(G)$ are dual to one another.

\subsection{The Burnside Ring}
For a finite group $G$, we denote by $\Fin_G$ the category of finite $G$-sets and equivariant functions between them. This has a coproduct given by disjoint union, and a product given by cartesian product. If we take isomorphism classes of objects, we obtain a semiring $(\pi_0\Fin_G^\simeq, \amalg, \times)$. Group completing yields the \emph{Burnside ring} of $G$, denoted $A(G):=K_0(\pi_0\Fin_G^\simeq, \amalg, \times)$.

To be explicit, an element of the Burnside ring of $G$ is a formal difference of isomorphism classes of finite $G$-sets. For convenience, we will not usually distinguish between a finite $G$-set and its isomorphism class -- for example, we may truthfully write $G/e \in A(G)$. When necessary for clarity, we will use square brackets to denote isomorphism classes, e.g. we may also write $-3[G/e] \in A(G)$.

For any subgroup $H\le G$, we obtain a \emph{mark homomorphism}, which is a ring homomorphism
\begin{align*}
    \phi^H \colon A(G) & \to \Z,
\end{align*}
defined by sending a finite $G$-set $X$ to the number $\lvert X^H \rvert$ of $H$-fixed points of $X$. As $H$ varies over conjugacy classes of subgroups of $G$, the mark homomorphisms jointly determine the Burnside ring. More explicitly:

\begin{proposition}[{\cite[5.1]{Dress-Bielefeld}}] Let $G$ be a finite group, and $C(G)$ the set of conjugacy classes of subgroups of $G$. Then
    \[
        \phi \coloneq (\phi^H)_{[H] \in C(G)} \colon A(G) \to \prod_{[H]\in C(G)}\Z
    \]
    is an injective ring homomorphism. Moreover, $\phi$ is a rational isomorphism.
\end{proposition}
As an abelian group, the Burnside ring is free of rank equal to the size of $C(G)$. As a ring, we understand its prime ideals very well:

\begin{proposition}[{\cite[5.2]{Dress-Bielefeld}}]
\label{prop:prime-ideals-AG}
For any (conjugacy class of) subgroup $H\le G$, and any prime number $p$, we define the ideals
    \begin{align*}
        I_{H,0} & := \ker (A(G) \xto{\phi^H} \Z)           \\
        I_{H,p} & := \ker (A(G) \xto{\phi^H} \Z \to \Z/p).
    \end{align*}
    Then every prime ideal in $A(G)$ is of the form $I_{H,0}$ or $I_{H,p}$ for some $H$ and $p$.
\end{proposition}

Given a finite $G$-set $X$, we can form its permutation representation. This yields the \emph{linearization map}
\[
    \theta \colon A(G) \to R_{\mathbb{C}}(G),
\]
which is in general neither injective nor surjective. The character of a permutation representation has a simple formula in terms of the mark homomorphisms:
\[\chi_{\theta(X)}(g) = \phi^{\langle g\rangle}(X),\]
where $\langle g\rangle$ is the cyclic subgroup generated by $g$. Thus we can describe the kernel of linearization in terms of marks:
\[
    \ker(\theta) = \bigcap_{\substack{C\le G \\ C\text{ cyclic}}} I_{C,0}.
\]
\subsection{Choosing \texorpdfstring{$G$}{G}-sets}
Let $X$ be a finite $G$-set and $k$ a natural number. Then we may consider the set of all $k$-element subsets of $X$, which we denote $\binom{X}{k}$, i.e.
\[\binom{X}{k} = \{S \subseteq X : \lvert S \rvert = k\}.\]
This has a natural $G$-action, given by
\[gS = \{gs : s \in S\}.\]
\begin{example}\label{ex:Ge-choose-2} Let $G[2] \subseteq G$ be the subset of elements of order $\le 2$. Then we have an isomorphism of $G$-sets:
    \begin{align*}
        \binom{G/e}{2} \cong \coprod_{i=1}^{ \frac{\#G - \#G[2]}{2}} G/e \amalg \coprod_{\substack{C\le G \\ |C|=2}}G/C.
    \end{align*}
\end{example}
\begin{proof} Given an unordered tuple $\{x,y\}$ of two elements in $G$, we can inquire about its stabilizer. If $g\in G$ satisfies $g\{x,y\} = \{x,y\}$, we find ourselves in one of two positions:
    \begin{enumerate}
        \item $gx=x$ and $gy=y$ so $g=e$,
        \item $gx=y$ and $gy=x$, so $g=yx^{-1}$ is 2-torsion.
    \end{enumerate}
    From this observation the result follows.
\end{proof}

We similarly have an analogue of ``$n$ choose $k$ with multiplicity,'' given by the symmetric power $X^{\times k}/\Sigma_k$.

\begin{notation}\label{nota:exterior-symmetric-on-AG}
    Given a finite $G$-set $X$ and some $k\ge 0$, we denote by $\epsilon$ and $\sigma$ the \emph{exterior} and \emph{symmetric} operations:
    \begin{align*}
        \epsilon^k(X) & = \binom{X}{k}           \\
        \sigma^k(X)   & = X^{\times k}/\Sigma_k.
    \end{align*}
    Observe when $k=0$, we obtain the singleton set $G/G$.
\end{notation}

On the representation ring $R(G)$, we similarly have exterior and symmetric operations, given by taking exterior and symmetric powers of representations, which we denote by $\epsilon$ and $\sigma$ as well.

\begin{proposition} Given a finite $G$-set $X$ and some $k\ge 0$, we have that linearization commutes with symmetric powers:
    \[
        \theta(\sigma^k(X)) \cong \sigma^k(\theta(X)),
    \]
    however linearization \emph{need not} commute with exterior powers.
\end{proposition}

This appears to be first noted by Siebeneicher in \cite{Siebeneicher}. A simple example where we can see the failure of linearization to commute with exterior products is the case of $C_2/e$ --- the permutation representation of $\epsilon^2(C_2/e)$ is trivial, but $\wedge^2 \C[C_2/e]$ is the sign representation.

The operations $\epsilon^k$ and $\sigma^k$ can be defined for virtual $G$-sets and virtual representation as well, hence they extend to functions from the Burnside ring or representation ring to itself. The collection of such operations is formalized by the idea of a \emph{pre-$\lambda$ structure}.

\subsection{Pre-\texorpdfstring{$\lambda$}{λ} Structures}

Pre-$\lambda$ structures were first defined by Grothendieck in his work on Chern classes, as a way to encode the common algebraic structures enjoyed both by $K$-groups and representation rings \cite[p.~148]{Grothendieck-Chern}. They further appeared in work of Atiyah and Tall, where they found connections to Adams operations and the $J$-homomorphism \cite[\S1]{AtiyahTall}.

\begin{definition}[{\cite[p.7]{Knutson}}] A \emph{pre-$\lambda$ ring} is a commutative unital ring $A$ together with a series of operations
    \begin{align*}
        \lambda^n\colon A \to A,
    \end{align*}
    for $n\ge 0$, satisfying
    \begin{enumerate}[label=(PL\arabic*)]
        \item $\lambda^0(x) = 1$
        \item $\lambda^1(x) = x$
        \item $\lambda^n(x+y) = \sum_{j=0}^n \lambda^j(x) \lambda^{n-j}(y)$.
    \end{enumerate}
    Rather than a ton of functions valued in $A$, we can repackage the entire data of the pre-$\lambda$ structure as a single function valued in formal power series. Explicitly, we define $\lambda_t$ as the following:
    \begin{equation}\label{eqn:lambda-t}
        \begin{aligned}
            \lambda_t \colon A & \to A\fps{t}                         \\
            x                  & \mapsto \sum_{n\ge0}\lambda^n(x)t^n.
        \end{aligned}
    \end{equation}
    With this new notation in mind, it is easy to see that PL3 can be rephrased as
    \begin{enumerate}[label=(PL\arabic*),start=3]
        \item $\lambda_t(x+y) = \lambda_t(x)\lambda_t(y)$.
    \end{enumerate}
\end{definition}

\begin{example} The following rings and $\lambda_t$'s define pre-$\lambda$ structures:
    \begin{enumerate}
        \item $\lambda_t(n) = (1+t)^n$ in $\Z\fps{t}$. This is called the \emph{binomial} pre-$\lambda$ structure.
        \item Given any formal power series $f(t)\in 1 + t\Z\fps{t}$, we have that $\lambda_t(n) = f(t)^n$ defines a pre-$\lambda$ structure on $\Z$.
        \item $\lambda_t(x) = e^{tx} \in \mathbb{R}\fps{t}$
    \end{enumerate}
\end{example}

\begin{definition}\label{def:category-pre-lambda} A \emph{homomorphism of pre-$\lambda$ rings} $f \colon (R,\lambda) \to (S, \mu)$ is a ring homomorphism $f \colon R \to S$ so that $f(\lambda^n(x)) = \mu^n(f(x))$ for any $n\ge0$ and any $x\in R$. This allows us to define a category $\prelring$.
\end{definition}

Over every commutative ring, pre-$\lambda$ structures come in pairs, given by the structure $\lambda$ and what we will call its \emph{dual} $\lambda^\star$, to be defined below. While this duality is well-known to experts (cf.~\cite[p.~412]{Hazewinkel-Witt}), we find it useful to spell out the key properties, in particular how it interchanges exterior and symmetric structures.

\begin{definition}\label{def:dual-pre-lambda-structure} Let $(A,\lambda_t)$ be any pre-$\lambda$ ring. We define the \emph{dual} pre-$\lambda$ structure $\lambda_t^\star$ on $A$ by the formula
    \begin{align*}
        (\lambda^\star)^n(x) = (-1)^n \lambda^n(-x)
    \end{align*}
    for any $x\in A$ and $n\ge0$. We remark that $\lambda_t^\star$ can be thought of as
    \begin{equation}\label{eqn:lambda-t-star-formulations}
        \begin{aligned}
            \lambda_t^\star(x) = \lambda_{-t}(-x) = \lambda_{-t}(x)^{-1},
        \end{aligned}
    \end{equation}
    where $\lambda_{-t}(y)$ is shorthand for $\lambda_t(y) \mid_{t \mapsto -t}$.
\end{definition}

\begin{proposition}\label{prop:star-duality} The assignment
    \begin{align*}
        (A,\lambda_t) \mapsto (A,\lambda_t^\star)
    \end{align*}
    defines a functor $\prelring \to \prelring$ which preserves underlying ring homomorphisms. This functor squares to the identity.
\end{proposition}
\begin{proof}
    We first check $\lambda_t^\star$ is indeed a pre-$\lambda$ ring structure. It is immediate to see that PL1 and PL2 hold. To check PL3 it is easier to use the first characterization:
    \begin{align*}
        (\lambda^\star)^n(x+y) & = (-1)^n \lambda^n(-x-y) = (-1)^n\sum_{j=0}^n \lambda^j(-x) \lambda^{n-j}(-y) \\
                             & = \sum_{j=0}^n (-1)^j \lambda^j(-x) (-1)^{n-j}\lambda^{n-j}(-y)               \\
                             & = \sum_{j=0}^n (\lambda^\star)^j(x) (\lambda^\star)^{n-j}(y).
    \end{align*}
    It is clear that $(-)^\star$ is involutive on objects, so we just verify functoriality. For any morphism $f \colon (R,\lambda) \to (S,\mu)$ of pre-$\lambda$ rings, and any $r\in R$, we have
    \[f(\lambda^\star_t(r)) = f(\lambda_{-t}(r)^{-1}) = f(\lambda_{-t}(r))^{-1} = \mu_{-t}(f(r))^{-1} = \mu^\star_t(f(r)),\]
    as desired.
\end{proof}

\begin{example}\label{exa:duality-on-Z}
    Let $\lambda_t$ be the binomial pre-$\lambda$ structure on $\Z$. Then its dual is given by
    \begin{align*}
        (\lambda^\star)^k(n) = \binom{n+k-1}{k}.
    \end{align*}
    That is, it counts \emph{with multiplicity}.
\end{example}
\begin{proof} As $(\lambda^\star)^k(n) = (-1)^k \lambda^k(-n)$, we just have to determine what is meant by the binomial of a negative number in the binomial pre-$\lambda$ structure. By PL1 and PL3, we have, for $k\ge 1$, that
    \begin{align*}
        0 =\lambda^k(n-n) = \sum_{j=0}^k \lambda^j(n) \lambda^{k-j}(-n),
    \end{align*}
    and we find inductively that $\lambda^k(-n) = (-1)^k \binom{n+k-1}{k}$.
\end{proof}

\begin{remark}
There is a stronger notion of a \emph{$\lambda$-ring}, which we will not define or use in this paper; we remark for the reader who is familiar with this definition that the duality $(-)^\star \colon \prelring \to \prelring$ does not preserve the subcategory of $\lambda$-rings. This can be seen, for instance, using uniqueness of the $\lambda$-structure on $\Z$ given by binomials, together with the previous example.
\end{remark}

\begin{example}\label{exa:duality-on-RG}
    Let $R(G)$ be the complex representation ring of a group $G$, and define the \emph{exterior} pre-$\lambda$ structure by
    \begin{align*}
        \lambda^n([V]) = \left[\bigwedge^n V\right].
    \end{align*}
    Then $(\lambda^\star)^n([V]) = [\Sym^n V]$.
\end{example}
\begin{proof} By \Cref{eqn:lambda-t-star-formulations} it suffices to see that $1 = \lambda_{-t}^\star(V) \lambda_t(V)$ for any $V$. The coefficient on $n\ge1$ in the right hand side of this expression is given by
    \begin{align*}
        \sum_{j=0}^n (-1)^j \Sym^j([V]) \otimes \bigwedge^{n-j}V,
    \end{align*}
    which vanishes in the representation ring (see for instance \cite[p.31]{Knutson}).
\end{proof}

Both \Cref{exa:duality-on-Z} and \Cref{exa:duality-on-RG} touch on a duality in mathematics between \emph{choosing} and \emph{choosing with repetition}. We would like to axiomatize this a bit more in the context of pre-$\lambda$ structures as it will be useful when investigating pre-$\lambda$ structures on the Burnside ring of a group. To that end, we make the following definition (cf.~\cite[16.4]{Hazewinkel-Witt}).

\begin{definition}\label{def:exterior-symmetric} Let $\Z_{\text{binom}}\in \prelring$ denote the integers equipped with the binomial pre-$\lambda$ ring structure. Let $\Z_{\text{multi}}$ denote the integers equipped with the multiset coefficients pre-$\lambda$ structure.
    \begin{enumerate}
        \item We refer to the slice category $\prelring_{/\Z_{\text{binom}}}$ as the category of \emph{exterior} pre-$\lambda$ structures.
        \item We refer to $\prelring_{/\Z_{\text{multi}}}$ as the category of \emph{symmetric} pre-$\lambda$ structures.
    \end{enumerate}
    Spelling this out, we say a pre-$\lambda$ ring $(R,\lambda_t)$ is \emph{exterior} if it comes equipped with an augmentation (a ring homomorphism) $R \to \Z$ which is a pre-$\lambda$ ring homomorphism when $\Z$ is equipped with the binomial structure (cf.~\cite[p.~20]{Knutson}).
\end{definition}

Clearly a pre-$\lambda$ ring can be either exterior or symmetric but not both, and many pre-$\lambda$ rings may not have either property, for instance if they don't admit an augmentation. An immediate observation is the following:

\begin{proposition} The duality $(-)^\star$ restricts to an isomorphism of categories
    \begin{align*}
        (-)^\star \colon \prelring_{/\Z_{\text{binom}}} \xto{\sim} \prelring_{/\Z_{\text{multi}}}.
    \end{align*}
\end{proposition}
In other words, a pre-$\lambda$ ring structure is exterior if and only if its dual is symmetric.

\begin{example} The exterior and symmetric powers on the representation ring $R(G)$ are, of course, exterior and symmetric, respectively. The augmentation exhibiting this is the natural one given by taking the rank of a representation.
\end{example}

\subsection{Pre-\texorpdfstring{$\lambda$}{λ} Structures on the Burnside Ring}
Recall the Burnside ring had these exterior and symmetric pre-$\lambda$ structures defined in \Cref{nota:exterior-symmetric-on-AG}. Linearization defines a pre-$\lambda$ ring homomorphism
\[
    \theta \colon (A(G),\sigma_t) \to (R(G),\sigma_t),
\]
therefore we can deduce the existence of two new pre-$\lambda$ structures from the failure of linearization to preserve exterior powers:
\begin{itemize}
    \item there is a \emph{dual exterior} pre-$\lambda$ structure $\epsilon_t^\star$ on $A(G)$,
    \item there is a \emph{dual symmetric} pre-$\lambda$ structure $\sigma_t^\star$ on $A(G)$.
\end{itemize}
These four pre-$\lambda$ structures $\epsilon_t$, $\epsilon_t^\star$, $\sigma_t$, $\sigma_t^\star$, are all different in general. The dual symmetric structure $\sigma_t^\star$ was first explored by Siebeneicher \cite{Siebeneicher}, where he proved the following theorem:

\begin{theorem}\label{thm:Siebeneicher}
    For a finite group $G$, we have that $\epsilon_t = \sigma_t^\star$ on $A(G)$ if and only if $|G|$ is odd.
\end{theorem}

We strengthen this result as follows:
\begin{theorem}\label{thm:Sieb-extension}
    Let $G$ be a finite group, then the following are equivalent:
    \begin{enumerate}
        \item $(A(G),\epsilon_t)$ and $(A(G),\sigma_t^\star)$ are isomorphic in $\prelring$
        \item $|G|$ is odd.
    \end{enumerate}
\end{theorem}
To prove this, we first define \emph{normalized} automorphisms of the Burnside ring, following \cite{kimmerleAutomorphismsBurnsideRings1995}.

\begin{proposition}\label{prop:normalized-aut-of-AG}
    Let $f \colon A(G) \xto{\sim} A(G)$ be a ring automorphism of the Burnside ring of a group. Then the following are equivalent:
    \begin{enumerate}
        \item $f(G/e) = G/e$
        \item $f(I_{e,0}) = I_{e,0}$, that is, $f$ preserves the augmentation ideal

        \item $f$ commutes with the augmentations to the integers, that is, the diagram commutes
              \[ \begin{tikzcd}
                      A(G)\ar[rr,"f" above]\ar[dr,"\phi^e" below left] &  & A(G)\ar[dl,"\phi^e" below right]\\
                      & \Z. &
                  \end{tikzcd} \]
    \end{enumerate}
    If any of these conditions hold, we say that $f$ is \emph{normalized}.
\end{proposition}
\begin{proof} Equivalence of the latter two is immediate. For equivalence with the first statement, we recall that any automorphism of the Burnside ring extends to an automorphism of the ghost ring, where it is obtained by permuting coordinates. Since $G/e$ has a single nonzero mark, its image $f(G/e)$ also has a single nonzero mark. If (3) holds, then the $e$-mark of $f(G/e)$ is equal to the $e$-mark of $G/e$ and all other marks are zero, therefore $f(G/e) = G/e$. Conversely if $f(G/e) = G/e$, then by extending $f$ to an automorphism on the Burnside ring, we see it fixes the coordinate corresponding to $e$-marks, which is equivalent to (3).
\end{proof}

It is a result of Kimmerle and Roggenkamp that if $C\le G$ is cyclic, and $f$ is a normalized automorphism of $A(G)$, then $f([G/C]) = [G/C']$ for some other cyclic subgroup $C'\le G$ \cite[4.1]{kimmerleAutomorphismsBurnsideRings1995}. We provide a self-contained proof in the case where $C$ is cyclic of order two, which is the only case needed to prove \Cref{thm:Sieb-extension}.

\begin{proposition}\label{prop:normalized-aut-preserves-GC-order-two}
    Let $G$ be a finite group and $f\colon A(G) \xto{\sim} A(G)$ a normalized automorphism. Then for any subgroup $C\le G$ of order two, we have that $f([G/C]) = [G/C']$ where $C'\le G$ is also of order two and $\lvert N_G(C) \rvert = \lvert N_G(C') \rvert$.
\end{proposition}
\begin{proof} Since $f$ extends to the ghost ring, $f([G/C])$ has exactly two nonzero marks. As $f$ is normalized, it preserves the $e$ mark, so
    \[
        \phi^e f([G/C]) = \lvert G \rvert/2.
    \]
    There is some other $H$ for which $\phi^H f([G/C]) = \phi^C ([G/C]) = \lvert N_G(C) \rvert / 2$, and all other marks beyond $e$ and $H$ are zero. Writing
    \[
        f([G/C]) = \sum_K a_K [G/K],
    \]
    we can pick a maximal subgroup $M\le G$ in the subconjugacy poset (a priori there could be multiple) for which $a_M\ne 0$. Taking $M$-marks, we get
    \[
        \phi^M  f([G/C])= \sum_K a_K \phi^M [G/K]
    \]
    Note that
    \[
        \phi^M([G/K]) \begin{cases}
            >0 & : M \leq K \text{ in the subconjugacy poset (i.e. $M$ is subconjugate to $K$)} \\
            =0 & : \text{else}.
        \end{cases}
    \]
    Therefore we get
    \[
        \phi^M f([G/C]) = \sum_{K\ge M} a_K \phi^M([G/K])
    \]
    If $K>M$ then $a_K = 0$ by hypothesis, therefore we get
    \[
        \phi^Mf([G/C]) = a_M \phi^M([G/M]) \ne 0,
    \]
    so $M=e$ or $M=H$ (in the subconjugacy poset). Therefore
    \[
        f([G/C]) = a_e[G/e] + a_H[G/H].
    \]
    If $e \lneq K \lneq H$, then $\phi^K(f([G/C])) = 0$ since there are only two nonzero marks, but then
    \[
        0 = \phi^K(a_e [G/e] + a_H [G/H]) = a_H \phi^K([G/H]).
    \]
    Since $\phi^K(G/H)\ne 0$, this implies $a_H=0$, but then $\lvert N_G(C) \rvert/2 = \phi^H(f([G/C])) = 0$ is a contradiction. Thus there are no proper nontrivial subgroups of $H$, hence $H$ is cyclic of some prime order $p$. Taking $e$ marks and using that we are normalized, we get
    \begin{align*}
        \phi^e f([G/C]) & = a_e|G| + a_{H} \frac{|G|}{p} \\
        = \phi^e([G/C]) & = \frac{|G|}{2}.
    \end{align*}
    Rearranging, this gives us
    \[
        2pa_e + 2a_{H} = p,
    \]
    hence $p$ is even, therefore equal to $2$. Now taking $H$-marks, we learn that $a_H$ is positive:
    \begin{align*}
        a_H \frac{\lvert N_G(H) \rvert}{2} = \phi^H(f([G/C])) = \frac{\lvert N_G(C) \rvert}{2} > 0 \implies a_H > 0.
    \end{align*}
    Using the fact that $f^{-1}$ is also a normalized automorphism of $A(G)$, we obtain
    \[[G/C] = f^{-1}(f([G/C])) = f^{-1}(a_e [G/e] + a_H [G/H]) = a_e [G/e] + a_H f^{-1}([G/H]),\]
    which we can rearrange as
    \[f^{-1}([G/H]) = \frac{1}{a_H} [G/C] - \frac{a_e}{a_H}[G/e].\]
    Since $f^{-1}([G/H])$ is an element of the Burnside ring, $\frac{1}{a_H}$ must be an integer. Since $a_H$ is a positive integer, we conclude that $a_H = 1$. It follows (by looking at $e$-marks again) that $a_e = 0$, and we conclude that $f([G/C]) = [G/H]$, as desired. Finally,
    \[\lvert N_G(H) \rvert = 2 \phi^H([G/H]) = 2 \phi^H(f([G/C])) = 2 \phi^C([G/C]) = \lvert N_G(C) \rvert.\qedhere\]
\end{proof}

\begin{lemma}\label{lem:PL-aut-of-AG-is-normalized}
    If $f \colon (A(G),\epsilon_t) \xto{\sim} (A(G),\sigma_t^\star)$ is an isomorphism of pre-$\lambda$ rings, then $f$ is a normalized automorphism of the Burnside ring.
\end{lemma}
\begin{proof} Let $A(G) \xto{\theta} R(G)$ denote the complex linearization map, sending a finite $G$-set to its permutation representation. Recall this is a map of pre-$\lambda$ rings, where $A(G)$ is equipped with the Siebeneicher (resp. symmetric) power structure and $R(G)$ is equipped with the exterior (resp. symmetric) power structure. As previously discussed, the kernel $K = \ker(\theta)$ of the complex linearization map is of the form
    \begin{align*}
        K = \bigcap_{j=1}^n I_{C_j,0},
    \end{align*}
    where $C_1, \ldots, C_n$ is a complete class of representatives for all the conjugacy classes of cyclic subgroups of $G$, and without loss of generality let's assume $C_1 = \{e\}$. The equality above holds because a representation is zero if and only if all its characters are zero, and these are just marks for cyclic groups.

    Note the following -- since the symmetric or exterior powers of the zero representation are still zero, and the complex linearization map is a map of pre-$\lambda$ rings, this implies that $K$ is closed under Siebeneicher operations.

    Now consider the composite
    \begin{align*}
        (A(G),\epsilon_t) \xto{f} (A(G),\sigma_t^\star) \xto{\theta} (R(G),\wedge_t).
    \end{align*}
    Then $f^{-1}(K)$ is some ideal in $A(G)$, which is closed under taking exterior powers. Moreover by \Cref{prop:prime-ideals-AG}, we have that $f^{-1}(K)$ is of the form
    \begin{align*}
        f^{-1}(K) = \bigcap_{j=1}^n I_{H_j,0},
    \end{align*}
    for some conjugacy classes of subgroups $H_j \le G$. We first claim that one of the $H_j$'s is equal to $\{e\}$. Suppose towards a contradiction none of them were. The finite $G$-set $G/e$ has no $H$-fixed points for any non-identity subgroup $H$, hence $[G/e] \in \bigcap_{j=1}^n I_{H_j,0}$. However it is clear that $\epsilon^{|G|}([G/e]) = [G/G]$ has nonzero marks for every group, hence cannot lie in $f^{-1}(K)$, which contradicts its closure under exterior operations.

    So now $H_j = \{e\}$ for some $j$. Since $f$ is an automorphism, we necessarily have that
    \begin{align*}
        f(I_{e,0}) = I_{C_r,0}
    \end{align*}
    for some other cyclic group $C_r$. We claim that $C_r = \{e\}$, which would tell us that $f$ preserves the augmentation ideal, and therefore is normalized by \Cref{prop:normalized-aut-of-AG}. Suppose towards a contradiction that $C_r \ne \{e\}$. Note that $I_{e,0}$ is closed under exterior powers, precisely because it is the kernel of the pre-$\lambda$ ring map to $\Z$ with the binomial structure (in general $I_{K,0}$ will not be closed under exterior powers unless $K=\{e\}$). Then by necessity, $I_{C_r,0}$ is closed under dual symmetric operations, and therefore symmetric operations. Note that since $C_r \ne \{e\}$ by hypothesis, we have that $G/e \in I_{C_r,0}$. However if $s = |C_r|$, then it is clear that $\sigma^s(G/e)$ has a $C_r$-fixed point given by a tensor corresponding to a coset of $C_r$ in $G$. This is a contradiction, therefore $C_r = \{e\}$ and $f$ is normalized.
\end{proof}

We can now prove our extension of Siebeneicher's theorem.

\begin{proof}[Proof of~\Cref{thm:Sieb-extension}]
    One direction is clear, so it suffices to argue if an isomorphism of pre-$\lambda$ rings exists then $|G|$ is odd. By \Cref{lem:PL-aut-of-AG-is-normalized} any isomorphism of pre-$\lambda$ rings $f \colon (A(G),\epsilon_t) \xto{\sim} (A(G),\sigma_t^\star)$ is normalized. If such an isomorphism existed, then we can extend it to a ring automorphism of $A(G)\fps{t}$ fixing $t$, and being an automorphism of pre-$\lambda$ rings translates into the condition that
    \begin{align*}
        f(\epsilon_t(x)) = \sigma_t^\star(f(x)) = \sigma_{-t}(f(x))^{-1}.
    \end{align*}
    In other words $f(\epsilon_t(x))\sigma_{-t}(f(x)) = 1$. Expanding this expression a bit, we see that
    \begin{align*}
        \left( 1 + f(x)t + f \left( \binom{x}{2} \right)t^2 + \ldots  \right) \left( 1 - f(x)t + \sigma^2(f(x)) t^2 - \cdots \right) =1
    \end{align*}
    In particular the coefficient on $t^2$ must vanish. Plugging in $x=G/e$ into this coefficient, we necessarily have that
    \begin{align*}
        [G/e]^2 - f \left( \binom{[G/e]}{2} \right) = \sigma^2([G/e]).
    \end{align*}
    We understand this binomial by \Cref{ex:Ge-choose-2}, and it's not hard to see that $\sigma^2([G/e]) = \binom{[G/e]}{2} + [G/e]$. Letting $G[2] \subseteq G$ denote the subset of $2$-torsion elements, including the identity, we can rearrange to obtain:
    \begin{align*}
        (\#G[2] -1) [G/e] & = \sum_{\substack{C \le G \\ \text{order }2}} [G/C] + f\left(\sum_{\substack{C \le G \\ \text{order }2}} [G/C]\right).
    \end{align*}
    We have
    \[\sum_{\substack{C \leq G \\ \text{order }2}} [G/C] = \sum_{(C)} [G : N_G(C)] [G/C]\]
    where the sum on the right-hand side is over conjugacy classes of subgroups of order two.
    By \Cref{prop:normalized-aut-preserves-GC-order-two}, $f$ permutes the classes $[G/C]$ while preserving their multiplicities $[G : N_G(C)]$. Thus we have that $f(\sum_C [G/C]) = \sum_C [G/C]$. Therefore the above becomes
    \begin{align*}
        (\#G[2]-1)[G/e] = 2 \sum_{(C)}[G : N_G(C)] [G/C].
    \end{align*}
    $[G/e]$ and the $[G/C]$'s are linearly independent in $A(G)$, so we conclude that $\#G[2] = 1$ and that $G$ has no subgroups of order two. These are both equivalent to $\#G$ being odd.
\end{proof}

\section{The Dress Map}\label{sec:dress}
\secttoc

The Grothendieck--Galois correspondence exhibits, for a field $k$, an equivalence of categories between the category of finite sets equipped with a continuous action of the profinite absolute Galois group of $k$, and the opposite category of finite \'etale $k$-algebras over $k$. For a finite Galois field extension $L/k$, we can input a finite $G$-set and output a finite \'etale $k$-algebra associated to it. This algebra has a natural \emph{trace form} on it, which we can think about as living in the Grothendieck--Witt ring of quadratic forms over $k$.\footnote{All of our fields have characteristic $\ne 2$.} We can see this at the level of spectra, where we will obtain a map from the $G$-fixed points of the $G$-equivariant sphere spectrum to the Hermitian $K$-theory spectrum of the ground field $k$.

At the level of $\pi_0$, this map was first studied by Dress \cite[Appendix~B]{Dress-Bielefeld}. We establish that it is a pre-$\lambda$ ring homomorphism for the exterior structure $\epsilon_t$ on $A(G)$, and a new pre-$\lambda$ structure $b_t$ we define on $\GW(k)$.

\textbf{Assumption}: For this section we denote by $k$ a field of characteristic $\ne 2$.

\subsection{Unimod and the Scharlau Map} We denote by $\Unimod_k$ the category whose objects are finite-dimensional $k$-vector spaces equipped with non-degenerate symmetric bilinear forms, and whose morphisms are isometries between these: explicitly, a morphism $f \colon (V,\alpha) \to (W,\beta)$ is a $k$-linear morphism $f\colon V \to W$ satisfying
\begin{align*}
    \beta(f(v_1),f(v_2)) = \alpha(v_1,v_2)
\end{align*}
for every pair of vectors $v_1,v_2\in V$. The category $\Unimod_k$ admits finite biproducts given by direct sum, where the bilinear form on $(V,\alpha)\oplus(W,\beta)$ is given by $\langle v_1 \oplus w_1, v_2 \oplus w_2\rangle = \alpha(v_1,v_2)+\beta(w_1,w_2)$. There is also a symmetric monoidal structure on $\Unimod_k$ given by the tensor product of $k$-vector spaces, whose bilinear form is defined by $\langle v_1 \otimes w_1, v_2 \otimes w_2\rangle = \langle v_1,v_2\rangle\cdot\langle w_1,w_2\rangle$. The symmetric monoidal structure distributes over direct sum.

Attached to any finite \'etale $k$-algebra $E$ is its \emph{trace form} $\tr_{E/k}$, which we will see is a contravariant assignment from finite \'etale $k$-algebras to $\Unimod_k$. This is a bit easier to describe via the Galois correspondence.

Let $G_s = \Gal(k^s/k)$ denote the absolute Galois group of $k$, and $\Fin_{G_s}$ the category of finite sets equipped with a continuous $G_s$-action and equivariant maps between them. The fundamental theorem of Galois theory states that there is an equivalence of categories $\Fin_{G_s} \simeq \FEt_k^\op$ given by sending a finite $G_s$-set $X$ to the $k$-vector space of $G_s$-equivariant maps from $X$ to $k^s$. This has an algebra structure given by pointwise multiplication in $k^s$.

Given a finite $G_s$-set $X$, we can define a bilinear form on $\Map_{G_s}(X,k^s)$ given by the formula
\begin{equation}\label{eqn:trace-form-on-map-X}
    \begin{aligned}
        \Map_{G_s}(X,k^s) \times \Map_{G_s}(X,k^s) & \to k                          \\
        (f,g)                                      & \mapsto \sum_{x\in X}f(x)g(x).
    \end{aligned}
\end{equation}
\begin{proposition} Let $X\in \Fin_{G_s}$ correspond to a finite \'etale algebra $E = \Map_{G_s}(X,k^s)$ along the Galois correspondence. Then the bilinear form in \Cref{eqn:trace-form-on-map-X} is the \emph{trace form} of $E$, given by sending a pair of elements in $E$ to the trace of multiplication by their product, considered as a $k$-linear endomorphism of $E$:
    \begin{align*}
        E \times E & \to k                                      \\
        (e_1,e_2)  & \mapsto \Tr_k(\text{mult}_{e_1\cdot e_2}).
    \end{align*}
\end{proposition}
\begin{proof}
    It suffices to check this in the case that $X$ is a transitive $G_s$-set, i.e. $X = G_s/H$ for some finite index open subgroup $H \le G_s$. In this case, $E = (k^s)^H$ is a finite separable field extension of $k$, and an equivariant map $f \colon G_s/H \to k^s$ encodes a choice of element in $k^s$ fixed by $H$ (the image of the identity coset $eH$) together with its conjugates under the Galois action. The sum of these elements is the trace of $f(eH)$.
\end{proof}

This assignment is functorial, provided we restrict to the maximal subgroupoids of $\Fin_{G_s}$ and $\Unimod_k$, and we refer to it as the \emph{Dress functor}:

\begin{proposition}\label{prop:Dress-functor} The assignment in \Cref{eqn:trace-form-on-map-X} yields a well-defined functor which we call the \emph{Dress functor}
    \begin{align*}
        D\colon \Fin_{G_s}^{\simeq} & \to \Unimod_k^{\simeq}     \\
        X                           & \mapsto \Map_{G_s}(X,k^s).
    \end{align*}
\end{proposition}

The groupoid core $\Fin_{G_s}^{\simeq} \subseteq \Fin_{G_s}$ inherits symmetric monoidal structures given by cartesian product and disjoint union, and the former distributes over the latter. Similarly, $\Unimod_k^{\simeq}$ inherits tensor products and direct sums with the same distributivity. The Dress functor above is a functor between tight symmetric bimonoidal categories, in the sense of \cite[2.1.2]{JohnsonYau}. The Dress functor is compatible with these structures:

\begin{proposition} The Dress functor is strong bimonoidal:
    \begin{align*}
        D \colon (\Fin_{G_s}^\simeq, \amalg,\times) \to (\Unimod_k^\simeq, \oplus, \otimes).
    \end{align*}
\end{proposition}
\begin{proof} Seeing that $D(X\amalg Y) \cong D(X)\oplus D(Y)$ is immediate, since trace forms on a coproduct are given by a block sum. For the other monoidal structure, there is clearly a vector space homomorphism
    \begin{align*}
        D(X) \otimes D(Y) \to D(X\times Y),
    \end{align*}
    defined by
    \begin{align*}
        \Map_{G_s}(X,k^s) \otimes \Map_{G_s}(Y,k^s) & \to \Map_{G_s}(X\times Y, k^s)                                  \\
        (f,g)                                       & \mapsto \left[ (x,y) \xmapsto{f\times g} f(x)\cdot g(y)\right].
    \end{align*}
    The fact that this defines an isomorphism of $k$-vector spaces is immediate by the Galois correspondence. It therefore suffices to show that this map is a well-defined morphism in $\Unimod_k$. For this, we compute directly for $f_1,f_2\in \Map_{G_s}(X,k^s)$ and $h_1,h_2\in \Map_{G_s}(Y,k^s)$, that:
    \begin{align*}
        \langle f_1 \otimes h_1, f_2\otimes h_2\rangle & =
        \left( \sum_{x\in X} f_1(x)f_2(x)\right) \left( \sum_{y\in Y}h_1(y)h_2(y)\right)                        \\
                                                       & = \sum_{(x,y)\in X\times Y} f_1(x) h_1(y)f_2(x) h_2(y) \\
                                                       & = \langle f_1 \times h_1, f_2 \times h_2\rangle.\qedhere
    \end{align*}
\end{proof}

By \cite[11.6.12,~11.6.14]{JohnsonYau}, the $K$-theory of the domain and codomain of the Dress functor are $E_\infty$-ring spectra. We then obtain:

\begin{corollary}\label{cor:dress-map}
    Applying Elmendorf--Mandell $K$-theory to the Dress functor yields a map of $E_\infty$-rings which we call the \emph{Dress map}:
    \[
        D \colon K(\Fin_{G_s}) \to \KO(k).
    \]
\end{corollary}

Suppose now that $L/k$ is a finite Galois field extension, corresponding to a normal subgroup $N\trianglelefteq G_s$ with quotient $G:=G_s/N$. This defines an embedding
\[
    \Fin_G \to \Fin_{G_s}
\]
(equivalently, the inclusion of the subcategory of finite $G_s$-sets on which $N$ acts trivially) which can easily be seen to preserve the monoidal structures above on groupoid cores. Therefore we can obtain a ``restricted'' Dress map
\[
    D_{L/k} \colon K(\Fin_G) \to \KO(k),
\]
and by an equivariant version of the Barratt--Priddy--Quillen theorem, the domain of this map is the $G$-fixed points of the equivariant sphere spectrum $\mathbb{S}_G^G$. The following is immediate:

\begin{proposition}\label{prop:pi0D}
    For a $G$-Galois extension $L/k$, the Dress map on $\pi_0$ is of the form
    \[
        \pi_0 D_{L/k} \colon A(G) \to \GW(k),
    \]
    and is a ring homomorphism.
\end{proposition}
It can also be seen that $D_{L/k}$ is nontrivial on $\pi_1$ in general. We reserve discussion of this fact for future work.

\subsection{Pre-\texorpdfstring{$\lambda$}{λ} Structures on the Grothendieck--Witt Ring}
A very natural question to ask is whether the Dress map in \Cref{prop:pi0D} is interpretable as a pre-$\lambda$ ring homomorphism, which we address in the next subsection. We first investigate which natural pre-$\lambda$ ring structures arise on the Grothendieck--Witt ring of a field.

There are immediate exterior and symmetric operations, which were explored by McGarraghy, and defined as follows:
\begin{proposition}[{\cite{mcgarraghyExteriorPowersSymmetric2002,mcgarraghySymmetricPowersSymmetric2005}}] We can define exterior and symmetric operations
    \begin{align*}
        \epsilon^k(\langle a_1,\ldots,a_n\rangle) & = \sum_{i_1 < i_2 < \cdots < i_k} \langle a_{i_1}\cdots a_{i_k}\rangle        \\
        \sigma^k(\langle a_1,\ldots,a_n\rangle)   & = \sum_{i_1 \le i_2 \le \cdots \le i_k} \langle a_{i_1}\cdots a_{i_k}\rangle,
    \end{align*}
    which yield pre-$\lambda$ structures on $\GW(k)$ which are dual to one another.
\end{proposition}

However these don't appear to be related to the Dress map in any natural way. Work of Pajwani--Pál defines an alternative pre-$\lambda$ structure $a_t$ which is compatible with the symmetric operations on finite \'etale $k$-algebras (equivalently, finite $G$-sets).

\begin{notation}[{\cite{PajwaniPal}}] For a field $k$ and a unit $\alpha \in k^\times$, we denote by $\tau_\alpha$ the element\footnote{This quantity $\tau_\alpha$ is denoted by $t_\alpha$ in \cite{PajwaniPal} --- we use $\tau$ so as not to conflate it with the variable $t$.}
    \begin{align*}
        \langle 2,\alpha\rangle - \langle 1,2\alpha \rangle \in \GW(k).
    \end{align*}
\end{notation}

\begin{theorem}[\cite{PajwaniPal}]
    \label{thm:PP}
    There is a well-defined pre-$\lambda$ structure $a_t$ on $\GW(k)$, defined on rank one forms by
    \begin{align*}
        a_t(\langle \alpha \rangle) = \frac{1}{1 - \langle \alpha \rangle t} + \tau_\alpha \frac{t^2}{(1-t)^3}.
    \end{align*}
    Moreover if $\text{char}(k)\ne 2$, then for any finite $G$-Galois extension $L/k$, the Dress map
    \begin{align*}
        \pi_0 D_{L/k} \colon (A(G),\sigma_t) \to (\GW(k),a_t)
    \end{align*}
    is a morphism of pre-$\lambda$ rings.
\end{theorem}

In attempting to contrast this result with the quadratic binomial coefficients emerging in quadratically enriched Gromov--Witten theory, we reckon with the same issue as before, namely that the exterior and symmetric structures on the Burnside ring are not dual. Nevertheless, we can define a new pre-$\lambda$ structure on $\GW(k)$ which is compatible with exterior operations!

\subsection{The Dress Map as a Pre-\texorpdfstring{$\lambda$}{λ} Ring Homomorphism}

\begin{definition}\label{def:bt-power-structure} We define a new pre-$\lambda$ structure on $\GW(k)$, denoted $b_t$, by the formula
    \begin{equation}\label{eqn:bt-in-terms-of-at}
        b_t(x) = \frac{a_t(x)}{a_{t^2}(x)}.
    \end{equation}
\end{definition}
On rank one forms, this is defined as:
\begin{align*}
    b_t(\langle \alpha \rangle) = \frac{(t+1)^2(\langle \alpha \rangle t^2-1 )( (1+ \tau_\alpha)t^2 - 2t + 1)}{( \langle \alpha \rangle t - 1)( (1+\tau_\alpha)t^4 - 2t^2 + 1)}.
\end{align*}

\begin{proposition}\label{prop:bt-is-pre-l-structure}
    The formula $b_t$ as above gives a well-defined pre-$\lambda$ structure on $\GW(k)$ for $\text{char}(k)\ne 2$.
\end{proposition}
\begin{proof} PL1-PL3 follow immediately from \Cref{eqn:bt-in-terms-of-at}.
\end{proof}

As a result, we obtain our main result:

\begin{theorem}\label{thm:main} Let $k$ be a field of characteristic $\ne 2$, and let $L/k$ be a finite Galois extension of $k$ with Galois group $G$. Then the Dress map
    \begin{align*}
        \pi_0 D \colon (A(G),\epsilon_t) \to (\GW(k),b_t)
    \end{align*}
    is a well-defined pre-$\lambda$ ring homomorphism.
\end{theorem}
\begin{proof} In characteristic not equal to $2$, we have by \Cref{thm:PP} that the Dress map is a morphism of pre-$\lambda$ rings with the symmetric and Pajwani--P\'al power structures:
    \[
        \pi_0 D \colon (A(G),\sigma_t) \to (\GW(k),a_t).
    \]
    We claim on the Burnside ring that $\epsilon_t = \sigma_t/\sigma_{t^2}$. For any finite $G$-set $X$ and any subgroup $H\le G$, we can compute the $H$-marks of $\epsilon_t(X)$ and $\sigma_t(X)$ as
    \begin{align*}
        \phi^H \epsilon_t(X) & = \prod_{O \in H\backslash X}(1 + t^{\# O})        \\
        \phi^H \sigma_t(X)   & = \prod_{O \in H\backslash X} (1 - t^{\# O})^{-1},
    \end{align*}
    from which the result follows.
\end{proof}

In quadratically enriched Gromov--Witten theory, Brugall\'e and Wickelgren have defined \emph{quadratically enriched binomial coefficients} for a finite \'etale $k$-algebra $E$. Explicitly, given $E\in \FEt_k$ corresponding to a finite $G_s$-set $X$ and $j\ge0$, they write $\binom{E/k}{j} \in \GW(k)$ to denote the trace form of the finite \'etale $k$-algebra associated to the $G_s$-set $\binom{X}{j} = \epsilon^j(X)$ \cite{BrugalleWickelgren,brugalleWelschingerWittInvariants2025}. Our pre-$\lambda$ structure, when evaluated on the trace form of $E$, therefore has the following interpretation in terms of quadratically enriched binomial coefficients:
\begin{align*}
    b_t(\tr_E) = \sum_{j\ge 0} \binom{E/k}{j} t^j.
\end{align*}
As an immediate corollary, we obtain a new way to see the following, which is a corollary of a result of Serre regarding Witt invariants \cite[29.2]{garibaldiCohomologicalInvariantsGalois2003}.

\begin{corollary}
For any field $k$ of characteristic $\ne 2$, pick any two \'etale $k$-algebras $E_1$ and $E_2$ with the property that their trace forms $\tr_{E_1}$ and $\tr_{E_2}$ are equal in $\GW(k)$. Then for any $j$, we have that
    \begin{align*}
        \binom{E_1/k}{j} = \binom{E_2/k}{j} \text{ in }\GW(k).
    \end{align*}
\end{corollary}

In fact, the pre-$\lambda$ structure $b_t$ turns out to be very computable on rank one forms, as it is eventually 4-periodic:

\begin{proposition}\label{prop:periodicity-of-bn}
    On rank one forms, our pre-$\lambda$ structure $b$ satisfies:
    \begin{align*}
        b^n(\langle \alpha \rangle) = \begin{cases}
                                          1                                           & n=0                          \\
                                          \langle \alpha \rangle                      & n=1                          \\
                                          \langle 2 \rangle - \langle 2\alpha \rangle & n\equiv 2\pmod{4},\ n\ge 2   \\
                                          \langle 2\alpha\rangle - \langle 2 \rangle  & n\equiv 3\pmod{4},\ n\ge 2   \\
                                          1 - \langle \alpha \rangle                  & n\equiv 0 \pmod{4},\ n\ge 2  \\
                                          \langle \alpha \rangle - 1                  & n\equiv 1 \pmod{4},\ n\ge 2.
                                      \end{cases}
    \end{align*}
\end{proposition}
\begin{proof} Plugging in the formulas from \Cref{thm:PP}, and after some minor simplification, we get that
    \begin{align*}
        b_t(\langle \alpha \rangle) = \frac{(t+1)^2(\langle \alpha \rangle t^2-1 )( (1+ \tau_\alpha)t^2 - 2t + 1)}{( \langle \alpha \rangle t - 1)( (1+\tau_\alpha)t^4 - 2t^2 + 1)}.
    \end{align*}
    The power series with the periodic expression in the proposition is
    \begin{align*}
         & 1 + \langle \alpha\rangle t + t^2 \left( \frac{(\langle 2 \rangle - \langle 2\alpha \rangle) + (\langle 2\alpha\rangle - \langle 2 \rangle)t + (1- \langle \alpha \rangle)t^2 + (\langle \alpha \rangle -1 )t^3}{1-t^4} \right) \\
         & = \frac{t^4 + (1 + \langle \alpha \rangle)t^3 + (2 + \tau_\alpha)t^2 + (1+\langle \alpha \rangle)t + 1}{t^3 + t^2 + t + 1}
    \end{align*}
    Cross-multiplying this guess with the honest formula for $b_t(\langle \alpha \rangle)$, we get a complicated polynomial expression in $t$, but it is straightforward to see that all of the coefficients vanish using the basic properties of $\tau_\alpha$ established in \cite[\S3]{PajwaniPal}.
\end{proof}

\subsection{Computations with quadratically enriched binomial coefficients} In this section we provide a short reproof of \cite[Theorem~1.1]{ChenWickelgren}. As in that paper, we write $1$ to mean $\langle 1\rangle$ and $n$ to mean $n\langle 1 \rangle$ in $\GW(k)$ for $n\in \mathbb{Z}$.

\begin{theorem}[{\cite[1.1]{ChenWickelgren}}] Let $q$ be a prime power of odd order, and let $u\in \GW(\mathbb{F}_q)$ be the unique rank one form not equal to the identity. For any $n,j\ge 0$, we have
    \begin{align*}
        \binom{\mathbb{F}_{q^n}/\mathbb{F}_q}{j} = b^j \tr_{\mathbb{F}_{q^n}/\mathbb{F}_q} = \begin{cases}
                                                                                                 \binom{n}{j}                                           & n\text{ odd}   \\
                                                                                                 \binom{n-1}{j} + \binom{n-2}{j-2} + u \binom{n-2}{j-1} & n\text{ even}.
                                                                                             \end{cases}
    \end{align*}
\end{theorem}
\begin{proof} If $n$ is odd, then the trace form of the extension $\mathbb{F}_{q^n}/\mathbb{F}_q$ is $n$, hence this follows directly from induction on the expression $b_t(1) = 1+t$. If $n$ is even, the trace form is $(n-1) + u$. Using PL3, we have
    \begin{align*}
        b^j((n-1) + u ) & = \sum_{i=0}^j b^i(n-1)b^{j-i}( u ) = \sum_{i=0}^j \binom{n-1}{i}b^{j-i}(u).
    \end{align*}
    Since over a finite field $\langle 1,2\alpha\rangle = \langle 2,\alpha \rangle$ for any $\alpha \in \mathbb{F}_q^\times$, we have that the periodicity on $b^n$ from \Cref{prop:periodicity-of-bn} becomes twofold periodicity for $n\ge 2$. With that in mind, the formula above becomes:
    \begin{align*}
        \binom{n-1}{j} + \binom{n-1}{j-1}u + (-1)^j (1-u)\sum_{i=0}^{j-2} (-1)^i \binom{n-1}{i}.
    \end{align*}
    We can use the binomial identity $\sum_{i=0}^{j-2} (-1)^i \binom{n-1}{i} = (-1)^{j-2}\binom{n-2}{j-2}$. Plugging this in above, we get
    \begin{align*}
        \binom{n-1}{j} + \binom{n-1}{j-1}u + (1-u)\binom{n-2}{j-2},
    \end{align*}
    which simplifies to the expression in the statement of the theorem by Pascal's identity.
\end{proof}

\subsection{The \texorpdfstring{$b_t$}{bt} as Grothendieck--Witt invariants}

Letting $\Et_n(k)$ denote the set of isomorphism classes of \'etale $k$-algebras of degree $n$, we can consider $\Et_n$ as a functor from fields to sets. Similarly, the Grothendieck--Witt ring and the Witt ring can be viewed as functors from fields to sets.

\begin{notation}
We denote by $\Inv^{\GW}_k(n)$ the set of natural transformations
\[
    \Et_n \Rightarrow \GW
\]
of functors $\Field_{k/} \to \Set$, and we write $\Inv_k(n)$ to be the set of natural transformations $\Et_n \Rightarrow \operatorname{W}$. We call these the sets of \emph{Grothendieck--Witt invariants} and \emph{Witt invariants}, respectively, and note that each has a natural (pointwise) ring structure.
\end{notation}

Reducing modulo the hyperbolic form yields a natural transformation $\GW \Rightarrow \operatorname{W}$, which induces a ring homomorphism
\[
    \Inv_k^{\GW}(n) \to \Inv_k(n).
\]
\begin{proposition} For any $n$ and any field $k$, there is a natural isomorphism
\[
    \Inv_k^{\GW}(n) \xto{\sim} \Inv_k(n) \times_{\Z/2} \Z,
\]
given by reduction modulo $h$ and rank.
\end{proposition}
\begin{proof} Recall that $\GW(k) \cong \operatorname{W}(k) \times_{\Z/2} \Z$, naturally in the field $k$. In other words, we have an isomorphism of functors $\GW \cong \operatorname{W} \times_{\underline{\Z/2}} \underline{\Z}$, where an underline denotes a constant functor. As $\Nat(\Et_n,-)$ is limit-preserving, it suffices to argue that $\Nat(\Et_n, \underline{S}) \cong S$ for an arbitrary set $S$. Take any $\eta \colon \Et_n \to \underline{S}$. Pick a field $K \in \Field_{k/}$ and $E\in \Et_n(K)$. Passing to a separable closure, we have that $E\otimes_K K^s \cong (K^s)^n$, so by naturality we have that $\eta_K(E) = \eta_{K^s}(E\otimes_K K^s) = \eta_{K^s}((K^s)^n)$. Since $k$ is initial in $\Field_{k/}$ we can see that we must have that $\eta_{K^s}((K^s)^n) \cong \eta_k(k^n)$. Therefore evaluation at $k^n$ gives us a bijection
$\mathrm{ev}_{k^n} \colon \Nat(\Et_n, \underline{S}) \xto{\sim} S.$
\end{proof}

With this in mind, we can rephrase a result of Serre to look at Grothendieck--Witt invariants rather than Witt invariants.

\begin{theorem}[{\cite[\S29]{garibaldiCohomologicalInvariantsGalois2003}}] For any field $k$ and any $n$, we have that $\Inv^{\GW}_k(n)$ is a free $\GW(k)$-module on $(\lfloor \frac{n}{2}\rfloor + 1)$-many generators
\begin{align*}
    \Inv^{\GW}_k(n) \cong \GW(k) \left\{\varepsilon^0 \Tr, \varepsilon^1 \Tr, \ldots, \varepsilon^{\lfloor n/2\rfloor}\Tr\right\},
\end{align*}
where $\varepsilon^i : \GW \Rightarrow \GW$ are the exterior power operations and $\Tr : \Et_n \Rightarrow \GW$ is the trace form.
\end{theorem}
\begin{proof}
The theorem of Serre gives the same statement for $\Inv_k$. The result for $\GW$ then follows from the previous proposition.
\end{proof}

A key idea in the proof of Serre's theorem is that any Grothendieck--Witt invariant is determined by its values on multiquadratic algebras \cite[\S29]{garibaldiCohomologicalInvariantsGalois2003}, by which we mean algebras isomorphic to
\[A = K_{a_1}\times \cdots \times K_{a_m}\times K^\epsilon,\]
where $a_i \in K$, $K_a := K[x]/(x^2-a)$, $m = \lfloor n/2 \rfloor$ and $\epsilon = n-2m \in \{0,1\}$. Thus we can describe other potential bases for $\Inv_k^{\GW}(n)$ by their values on multiquadratic extensions. It is convenient to express these bases in terms of generating functions. We have different bases from \cite{brugalleWelschingerWittInvariants2025}, defined on multiquadratic algebras $A$ as follows:
\begin{enumerate}
    \item The $\lambda_j$'s, defined as the $j$th exterior power applied to the trace form of $A$. The generating function is
    \[
    \lambda_t(A) = (1+t)^\epsilon (1+\langle 2 \rangle t)^m \prod_{j=1}^m (1 + \langle 2 a_j\rangle t).
    \]
    \item The $\beta_j$'s, defined as the $j$th elementary symmetric polynomial of the trace form of the $K_{a_j}$'s. The generating function is
    \[
        \beta_t(A) = \prod_{j=1}^m (1 + \Tr(K_{a_j})t) = \prod_{j=1}^m (1 + \langle 2,2a_j\rangle t)
    \]
    \item The $\alpha_j$'s, defined as the $j$th elementary symmetric polynomial on the $\langle a_j\rangle$'s. Its generating function is
    \[
        \alpha_t(A) = \prod_{j=1}^m (1 + \langle a_j \rangle t).
    \]
    \item The $\chi_j$'s, defined as the trace forms of the $j$th symmetric powers of $A$.
\end{enumerate}
All of these extend to well-defined Witt invariants and yield bases of $\Inv_k(n)$ by \cite[2.5]{brugalleWelschingerWittInvariants2025} (and \cite[Appendix~A]{brugalleWelschingerWittInvariants2025} for the $\chi_j$'s).

Here we can apply our results to the idea floated in \cite[Remark~A.2]{brugalleWelschingerWittInvariants2025} to produce another basis for the ring of Grothendieck--Witt invariants.

\begin{definition}\label{def:Bj-GW-invariants}
For any $j\ge 0$, we define $B_j \in \Inv_k^{\GW}(n)$ to be $B_j := b^j\circ \Tr$, where $\Tr$ is the trace form and $b^j$ is as in \Cref{def:bt-power-structure}. This is the $j$th \emph{quadratically enriched binomial coefficient}.
\end{definition}

\begin{example}\label{exa:bt-quadratic-algebra}
Let $k_a = k[x]/(x^2-a)$ be a quadratic \'etale algebra over a field $k$. Then the generating function for the $B_j$'s satisfies:
\[
b_t (\Tr(k_a)) = 1 + \Tr(k_a)t + t^2 \in \GW(k)\fps{t}.
\]
\end{example}
\begin{proof} The first two coefficients are immediate, and the last follows from commuting $b_t$ with the trace form via \Cref{thm:main}.
\end{proof}

\begin{theorem}\label{thm:Bjs-basis-of-inv-gw}
Let $k$ be any field of characteristic $\ne2$, and $n\ge1$ any integer. Then we have an isomorphism of $\GW(k)$-modules
\begin{align*}
    \Inv^{\GW}_k(n) \cong \GW(k) \left\{ B_0,\ldots, B_{\lfloor n/2 \rfloor} \right\},
\end{align*}
where $B_j$ is the $j$th quadratically enriched binomial coefficient.
\end{theorem}
\begin{proof} In order to prove the result, we evaluate on a multiquadratic extension, and compute the matrix relating the $B_j$'s with the $\beta_j$'s and argue it is invertible (in fact, invertible over the integers).

On a multiquadratic extension $A = K_{a_1} \times \cdots \times K_{a_m} \times K^\epsilon$, with $m = \lfloor n/2\rfloor$ and $\epsilon = n-2m \in \{0,1\}$, we can rearrange the generating function for the $\beta$-basis:
\begin{align*}
    \beta_t(A) &= \prod_{i=1}^m (1 + \Tr(K_{a_i})t) = \prod_{i=1}^m (1 + \langle 2\rangle (1 + \langle a_i \rangle)t).
\end{align*}
Writing out the generating function for the binomial coefficients on a multiquadratic algebra, we observe that we can rewrite it in terms of $\beta_t$ as follows:
\begin{align*}
    b_t(\Tr(A)) &= (1+t)^\epsilon \prod_{j=1}^m \left(1 + \Tr(K_{a_j}) t + t^2\right) \\
    &= (1+t)^\epsilon \prod_{j=1}^m (1 + t^2)\left(1 + \Tr(K_{a_j}) \frac{t}{1+t^2}\right) \\
    &= (1+t)^\epsilon (1+t^2)^{m} \prod_{j=1}^m \left(1 + \Tr(K_{a_j}) \frac{t}{1+t^2}\right) \\
    &= (1+t)^\epsilon (1+t^2)^{m} \beta_{t/(1+t^2)} \\
    &= (1+t)^\epsilon (1+t^2)^m \sum_{j=0}^\infty \beta_j(A) \left( \frac{t}{1+t^2}\right)^j \\
    &= (1+t)^\epsilon (1+t^2)^m \sum_{j=0}^\infty \beta_j(A) \left( t - t^3 + t^5 - \ldots \right)^j
\end{align*}
Therefore the $B_i$'s are expressible in terms of the $\beta_j$'s. If we write $B_i = \sum_j N_{ij} \beta_j$ for some integers $N_{ij}$, it is clear that $N_{ij} = 0$ for $j>i$ by consideration of degrees. Similarly by consideration of degree when $j=i$, we have that $N_{ii} = 1$. Therefore the matrix relating the $B_i$-basis to the $\beta_j$-basis is integral and lower-triangular with ones along the diagonal, hence invertible.
\end{proof}

\begin{remark} A direct formula for the change of basis between the $B_i$'s and the $\beta_j$'s can be derived from the proof of \Cref{thm:Bjs-basis-of-inv-gw}. Explicitly, we have that
\[
    B_t = \sum_{\substack{p,q \geq 0 \\ p+q \leq m \\ 0 \leq f \leq \varepsilon}} \binom{m-p}{q}\binom{\epsilon}{f} \beta_p t^{p+2q+f}.
\]
This comes from expanding the expression
\[
    B_t(A) = (1+t)^\epsilon \prod_{j=1}^m (1 + \Tr(K_{a_j})t + t^2).
\]
In the expanded product, we pick $p$ different linear terms, $q$ different quadratic terms, and potentially one linear term coming from $(1+t)^\epsilon$, depending on whether $\epsilon=1$. Summing over all choices, we obtain the $p$th elementary symmetric polynomial on the $\Tr(K_{a_j})$'s, which is precisely $\beta_p(A)$. The $q$ choices come from the complement of these $p$ choices, hence the binomial $\binom{m-p}{q}$.
\end{remark}

We can furthermore relate the $B_j$'s to the $\chi_j$'s, and they admit a very elegant relation:

\begin{proposition} Let $A$ be any multiquadratic algebra. Then we have an equality of generating functions
\[
    B_t(A) = \frac{\chi_t(A)}{\chi_{t^2}(A)}.
\]
\end{proposition}
\begin{proof}
    This is exactly the relationship $\epsilon_t = \sigma_t / \sigma_{t^2}$ on the Burnside ring, followed by the Dress map.
\end{proof}

\begin{remark} While the coefficients on Welschinger--Witt invariants ``explode'' when written in the $\chi$-basis \cite[A.2]{brugalleWelschingerWittInvariants2025}, when written in the $B$-basis the coefficients appear to be a few orders of magnitude smaller. In particular we can rewrite Table 6 in the mentioned paper:
\begin{center}
\begin{tabular}{|r | p{6cm} | p{6cm}|}
    \hline
    $d$ & $\VW_d(\mathbb{P}^2)$ in $\chi$-basis & $\VW_d(\mathbb{P}^2)$ in $B$-basis \\
    \hline
    \hline
4 & $-2\chi_0 +2\chi_1 - \chi_2+\chi_3$

    & $-13B_0 + 3B_1+B_2 + B_3$
\\
\hline
	5& $-118\chi_0 -18\chi_1+ 18\chi_2+ \chi_3-\chi_4-\chi_5+\chi_6$

    & $-97B_0 - 60B_1 - 9B_2 + 9B_4 + 4B_5 + B_6$
    \\
        \hline
	   6& $-4474\chi_0 -9460\chi_1 + 6644\chi_2 + 828\chi_3-  2680\chi_4 + 836\chi_5 + 284\chi_6-236  \chi_7   +46\chi_8$

       & $2054 B_0 + 310 B_1 - 194 B_2 - 586 B_3 - 306 B_4 - 26 B_5 + 110 B_6 + 86 B_7  + 46 B_8$
       \\
       \hline
      7 & $-4519048\chi_0+  6205190 \chi_1  -448528\chi_2+  -2536404 \chi_3+  1184776\chi_4+  350017 \chi_5 -437346\chi_6 + 76967\chi_7 + 46664\chi_8 - 23786\chi_9 + 3336\chi_{10}$
      & $-50248 B_0 + 42 B_1 + 46160 B_2 + 35570 B_3 + 888 B_4 - 18177 B_5 - 16416 B_6 - 5223 B_7 + 3160 B_8 + 6238 B_9 + 3336 B_{10}$

	  \\
      \hline
      8 & $\approx 3.05 \cdot 10^9 \chi_0  -1.08\cdot 10^{10} \chi_1 +  3.59 \cdot 10^9\chi_2 +  2.94\cdot 10^9 \chi_3   -2.56 \cdot 10^9\chi_4 + 1.31 \cdot 10^8\chi_5 + 6.25 \cdot 10^8 \chi_6  -2.64\cdot 10^8 \chi_7 +  -1.48\cdot 10^7 \chi_8 + 4.16 \cdot 10^7\chi_9 -1.24 \cdot 10^7\chi_{10} +  1.29\cdot 10^6 \chi_{11}$
      & $4563970B_0-4602150B_1-6235818B_2-1498050B_3+2884542B_4+3320854B_5+888250B_6-1422942B_7-1943614B_8-952198B_9+418470B_{10}+1285806B_{11}$\\
\hline
\end{tabular}
\end{center}
We can observe the appearance of the unitriangular change of basis matrix in the fact that the leading coefficient is always the same.
\end{remark}

\printbibliography{}

\end{document}